\documentclass[preprint,12pt]{elsarticle}
\usepackage{amsthm,amsfonts,amssymb,amscd,amsmath,enumerate,verbatim,calc,graphicx,geometry}
\usepackage[all]{xy}
\usepackage{tikz-cd}
\newtheorem{theorem}{Theorem}[section]
\newtheorem{lemma}[theorem]{Lemma}
\newtheorem{proposition}[theorem]{Proposition}
\newtheorem{corollary}[theorem]{Corollary}

\newtheorem{example}[theorem]{Example}
\theoremstyle{definition}
\theoremstyle{definitions}
\newtheorem{definition}[theorem]{Definition}

\newtheorem{remark}[theorem]{Remark}
\theoremstyle{notations}

\theoremstyle{remarks}

\journal{ }

\begin{document}

\begin{frontmatter}



\title{Small Loop Transfer Spaces with Respect to Subgroups of Fundamental Groups}


\author[]{S.Z. Pashaei}
\ead{Pashaei.seyyedzeynal@stu.um.ac.ir}
\author[]{B. Mashayekhy\corref{cor1}}
\ead{bmashf@um.ac.ir }
\author[]{H. Torabi}
\ead{h.torabi@ferdowsi.um.ac.ir }
\author[]{M. Abdullahi Rashid}
\ead{mbinev@stu.um.ac.ir }
\address{Department of Pure Mathematics, Center of Excellence in Analysis on Algebraic Structures, Ferdowsi University of Mashhad,\\
P.O.Box 1159-91775, Mashhad, Iran.}
\cortext[cor1]{Corresponding author}
\begin{abstract}
 Let $H$ be a subgroup of $\pi_{1}(X,x_{0})$. In this paper, we extend the concept of $X$ being SLT space to $H$-SLT space at $x_0$. First, we show that the fibers of the endpoint projection $p_{H}:\tilde{X}_{H}\rightarrow X$ are topological group when $X$ is $H$-SLT space at $x_0$ and $H$ is a normal subgroup. Also, we show that under these conditions the concepts of homotopically path Hausdorff relative to $H$ and homotopically Hausdorff relative to $H$ coincide. Moreover, among other things, we show that the endpoint projection map $p_{H}$ has the unique path lifting property if and only if $H$ is a closed normal subgroup of  $\pi_{1}^{qtop}(X,x_{0})$ when $X$ is SLT at $x_{0}$. Second, we present conditions under which the whisker topology is agree with the quotient of compact-open topology on $\tilde{X}_{H}$. Also, we study the relationship between open subsets of $\pi_{1}^{wh}(X,x_{0})$ and $\pi_{1}^{qtop}(X,x_{0})$.

\end{abstract}


\end{frontmatter}



\section{Introduction and Preliminalies}

Throughout this article, we consider a path connected topological space X with a base point  $x_{0}\in{X}$. Let $H$ be a subgroup of $\pi_{1}(X,x_{0})$.
We recall the set $\tilde{X}_{H}$ as follows. Consider the equivalence relation  $\sim_{H}$ on the set of all paths in $X$ starting at $x_{0}$, $P(X,x_{0})$, by $\alpha \sim_{H} \beta$ if $\alpha(1)=\beta(1)$ and $[\alpha\ast\beta^{-1}]\in{H}$. Then $\tilde{X}_{H}$ is the set of all equivalence classes of based paths such as $\alpha$ in $X$ under the relation $\sim_{H}$ which is denoted by $[\alpha]_{H}$. The points of $\tilde{X}_{H}$ are said to be $H$-equivalent classes.

  A classical way to introduce a topology on $\tilde{X}_{H}$ is as the quotient space of the space $P(X,x_{0})$ equipped with the compact-open topology. We denote the space $\tilde{X}_{H}$ equipped with this topology by $\tilde{X}^{top}_{H}$. Another famous classical topology on $\tilde{X}_{H}$ which is introduced by Spanier \cite[proof of Theorem 13, Section 5, Chapter 2]{Spanier}  and named by Brodskiy et al.  \cite{Bcovering} the whisker topology,  is as follows (see also \cite[Section 2.1]{Bog}  and \cite[Section 2]{Zastrow}).

 \begin{definition}
For any pointed topological space $(X, x_{0})$ the whisker topology on the set $\tilde{X}_{H}$ is defined by the basis $([\alpha]_{H}, U) = \lbrace[\alpha\ast\beta]_{H}\rbrace$, where $\alpha$ is a path in $X$ from $x_{0}$ to $x_{1}$, $U$ is a neighborhood of $x_{1}$ in X, and $\beta$ is a path in $U$ originating at $x_{1}$.
 \end{definition}

 In this topology a neighborhood of a $H$-equivalent class of a path consists of all paths obtained by prolonging the original path by a small amount.  We denote the space $\tilde{X}_{H}$ equipped with the whisker topology by $\tilde{X}^{wh}_{H}$. The endpoint projection $p_{H}:\tilde{X}^{wh}_{H}\rightarrow X$ assigns to each $H$-equivalence class the end point of its path. We denote $\tilde{X}^{top}_{e}$, $\tilde{X}^{wh}_{e}$ (which is called \textit{universal path space }\cite[Definition 2]{ZVirk}) and $p_{e}$ instead of $\tilde{X}^{top}_{H}$, $\tilde{X}^{wh}_{H}$ and $p_{H}$, respectively, when $H$ is the trivial subgroup. Note that $\pi_{1}^{qtop}(X,x_{0})$ and $\pi_{1}^{wh}(X,x_{0})$ can be considered as the subspaces of $\tilde{X}^{top}_{e}$ and $\tilde{X}^{wh}_{e}$, respectively. Also, the fiber $p^{-1}_{e}(x_{0})\subset \tilde{X}^{wh}_{e}$ can be identified with the fundamental group $\pi_{1}^{wh}(X,x_{0})$. In \cite[Proposition 4.20]{BroU} it is shown that  $p^{-1}_{e}(x_{0})$ is a topological group under some conditions. In this paper, we address to more general case of this result. Indeed, we address the question of whether the induced topologies on $p_{H}^{-1}(x_{0})$ from $\tilde{X}^{wh}_{H}$ and $\tilde{X}^{top}_{H}$ make it a topological group. In the case $H=1$, it fails to be a topological group (see \cite{BroU}). Note that $p^{-1}_{H}(x_{0})\subset \tilde{X}_{H}$ forms a group as long as $H$ is a normal subgroup.

 The concept of small loop transfer space which have been introduced and studied by Brodskiy et al. \cite{BroU} is defined as follows:

\begin{definition}
A topological space $X$ is called a small loop transfer (SLT for short) space at $x_{0}$ if for every path $ \alpha $ in $X$ with $\alpha(0)=x_{0}$  and for every neighborhood $U$ of $x_{0}$ there is a neighborhood $V$ of $\alpha(1)=x$ such that for every loop $\beta$ in $V$ based at $x$ there is a loop $\gamma$ in $U$ based at $x_{0}$  which is homotopic to $ \alpha\ast\beta\ast\alpha^{-1}$ relative to $\dot{I}$. The space $X$ is called an SLT space if $X$ is SLT at $x_{0}$ for every $x_{0 }\in{X}$.
\end{definition}

We are going to extend the approach of Section 4 in \cite{BroU} in order to define relative version of small loop transfer spaces with respect to subgroups of the fundamental group.  First, we define the concept $H$-small loop transfer space at $x_{0}$ and then $H$-small loop transfer space ($H$-SLT for abbreviation), where $H$ is a subgroup of $\pi_{1}(X,x_{0})$. We show that for an $H$-SLT space at $x_{0}$ the fiber $(p^{-1}_{H}(x_{0}))^{wh}$, which is agree with $ \frac{\pi_{1}^{wh}(X,x_{0})}{H} $, is a topological group and also $(p^{-1}_{H}(x_{0}))^{top}$ is a topological group when $X$ is locally path connected. As a result of this statement, $\pi_{1}^{wh}(X,x_{0})$ is a topological group if $X$ is small loop transfer at $x_{0}$.

The classical theory of covering spaces is based on construction of the universal covering space for
a space $X$ as a space of homotopy classes of certain paths. In the case of path connected, locally path
connected, semilocally simply connected spaces there is an obvious choice of topology on such space, which
makes it the universal covering space (see \cite{Spanier}). Many people have attempted to extend the covering-theoretic approach to more general spaces (see e.g. \cite{Beres, BrazS, Baction, Fox, Lubkin}). For instance, Fox's overlays \cite{Fox} provide no more information about the subgroup lattice of $\pi_{1}(X,x_{0})$ than traditional covering maps but admit a much more general classification in terms of the fundamental pro-group. Then, Brazas \cite[Definition 3.1]{BrazS} introduces the notion of semicovering map as a local homeomorphism with continuous lifting of paths and homotopies and $\textbf{lpc}_{0}$-covering maps (see \cite[Definition 5.3]{BrazG}) which are defined in terms of unique lifting properties. With the exception of locall triviality, semicoverings enjoy nearly all of the important properties of coverings. $\textbf{lpc}_{0}$-covering maps often exist when standard covering maps do not and provide combinatorial information about fundamental groups of spaces which are not semilocally simply connected. Recall that generalized universal covering maps, which are certain case of $\textbf{lpc}_{0}$-covering maps, for the first time is introduced by Fischer and Zastrow \cite{Zastrow}. For example, one-dimensional spaces such as the \textit{Hawaiian earring} (denoted by HE) admit the generalized universal covering.

The endpoint projection map $p_{H}$ is $\textbf{lpc}_{0}$-covering map if and only if it has the unique path lifting property \cite[Lemma 5.9]{BrazG}. Moreover, Brazas in \cite[Lemma 5.10]{BrazG} has verified that any $\textbf{lpc}_{0}$-covering map $p:\tilde{X}\rightarrow X$ is equivalent to the endpoint projection map $p_{H}$, where $H=p_{\ast}\pi_{1}(\tilde{X},\tilde{x}_{0})$ for $\tilde{x}_{0}\in{p^{-1}(x_{0})}$.
Hence, the map $p_{H}: \tilde{X}^{wh}_{H}\rightarrow X$ is a good candidate for study any $\textbf{lpc}_{0}$-covering subgroup $H$ of $\pi_{1}(X,x_{0})$.

 The concepts of \textsf{homotopically Hausdorff relative to $H$} and \textsf{homotopically path Hausdorff relative to $H$} are related to $p_{H}$ which are appeared in \cite{Zastrow, BrazFa}.
 \begin{proposition}
 Let $(X, x_{0})$ be a path connected space.
 \begin{itemize}


 \item[(i)] If the endpoint projection $p_{H}$ has the unique path lifting property then $X$ is homotopically Hausdorff relative to $H$.

\item[(ii)] If $X$ is homotopically path Hausdorff relative to $H$ then the endpoint projection $p_{H}$ has the unique path lifting property.
 \end{itemize}
 \end{proposition}

In the case $H=1$, Virk and Zastrow gave an example of $X$ in which $X$ is homotopically Hausdorff but it does not admit the generalized universal covering space \cite{VZexam}. Therefore these two concepts are not equivalent, in general. In Section 2, we show that if $X$ is $H$-SLT at $x_{0}$, then these two concepts are equivalent, when $H$ is a normal subgroup. Indeed, we find a condition under which we can answer to the converse of parts (i) and (ii) of the above proposition.

 It is known that fibers of semicovering maps are discrete \cite{BrazO}. Note that this property does not hold for a $\textbf{lpc}_{0}$-covering map and even more, its fibers are not homeomorphic \cite[Example 4.15]{Zastrow}.  In Corollary 2.14, we show that fibers of a $\textbf{lpc}_{0}$-covering map are homeomorphic in SLT spaces. Finally in  Section 2, Corollary 2.16 shows that a semicovering map can transmit the property of being SLT from its codomain to its domain.


The last section can be viewed as generalization of results announced by Brodskiy et al. \cite{BroU}. They investigated relationship between $\tilde{X}^{wh}_{e}$ and $\tilde{X}^{top}_{e}$ and showed that these two spaces are identical when $X$ is an SLT space. It seems interesting to investigate relationship between $\tilde{X}^{wh}_{H}$ and $\tilde{X}^{top}_{H}$. We show that the whisker topology and the compact-open topology on $\tilde{X}_{H}$ are identical when the space $X$ is $H$-SLT (see Theorem 3.1).

Recall that  Virk and Zastrow in \cite{VZcom} studied $\pi_{1}^{wh}(X,x_{0})$ and $\pi_{1}^{qtop}(X,x_{0})$ and stated that the whisker topology is in general strictly finer than the compact-open topology on $\pi_{1}(X,x_{0})$. Brodskiy et al. \cite{BroU} verified that all open subsets of the whisker topology and the compact-open topology are the same in SLT spaces. On the other hand, Brazas \cite{BrazT} showed that $\pi_{1}^{\tau}(X,x_{0})$, which is a quotient space of the free (Markov) topological group of $\pi_{1}^{qtop}(X,x_{0})$, and $\pi_{1}^{qtop}(X,x_{0})$ have the same open subgroups while this property does not hold for $\pi_{1}^{qtop}(X,x_{0})$ and $\pi_{1}^{wh}(X,x_{0})$ since by unifying their open subgroups one can imply that those are coincidence. In Section 3, we show that all open subsets of $\pi_{1}^{wh}(X,x_{0})$ and $\pi_{1}^{qtop}(X,x_{0})$  containing a normal subgroup $H$ are the same when $X$ is $H$-SLT at $x_{0}$. Also, in the case that $H$ is not normal, open subgroups of $\pi_{1}^{wh}(X,x_{0})$ and $\pi_{1}^{qtop}(X,x_{0})$ containig $H$ are the same. Moreover, we verify that fibers of a $\textbf{lpc}_{0}$-covering map are Hausdorff. It is known that semicovering maps are $\textbf{lpc}_{0}$-covering maps but not vice versa. In Corollary 3.9, we prove that the converse holds if a fiber of the $\textbf{lpc}_{0}$-covering map is finite when the domain is an $H$-SLT space. Finally, we give an example of a space $X$ which is $H$-SLT but it is not an SLT space for a special subgroup $H$.




\section{Relationship Between SLT Spaces and the Endpoint Projection Map}

For any pointed topological space $(X, x_{0})$, let $\tilde{X}^{top}_{H}$ be the quotient space of the path space $P(X,x_{0})$ equipped with the compact-open topology. We denote by $(p_{H}^{-1}(x_{0}))^{top}$ which its topology will be induced from $\tilde{X}_{H}^{top}$. We denote $\tilde{X}_{H}$ with the whisker topology by  $\tilde{X}_{H}^{wh}$. The subspace topology which is induced on $p_{H}^{-1}(x_{0})$ from $\tilde{X}_{H}^{wh}$  is denoted by $(p_{H}^{-1}(x_{0}))^{wh}$.

It can be defined easily an action on the fiber $ p^{-1}_H(x_0) $ as follows provided that $H$ is a normal subgroup of $\pi_{1}(X,x_{0})$:
$\theta: p^{-1}_H(x_0) \times p^{-1}_H(x_0)\rightarrow p^{-1}_H(x_0)$, where $ \theta([\alpha] _{H}, [\beta]_{H})=[\alpha\ast\beta]_{H} $. Indeed, $\theta$ is well-defined when $H$ is a normal subgroup of $\pi_{1}(X,x_{0})$. Therefore,  action $\theta$ makes $ p^{-1}_H(x_0) $ be equipped with a group structure.  Note that property of being normal of $H$ is a necessary and sufficient condition to $ (p^{-1}_{H}(x_{0}))$   turn out to be a group. In relation to $p_{e}^{-1}(x_{0})$ , it is a group because the trivial subgroup $H=1$ is a normal subgroup. Recall that $\pi_{1}(X,x_{0})=p_{e}^{-1}(x_{0})$.


In \cite[Proposition 4.20]{BroU}, it is shown that if the operation of taking inverse in $ \pi _{1}^{wh}(X,x_0)$
$=(p_{e}^{-1}(x_{0}))^{wh} $ is continuous, then $ {{\pi }_1}^{wh}(X,x_0) $ is a topological group. Similarly, we have the following proposition for $ (p^{-1}_H(x_0))^{wh} $.


\begin{proposition}
Let $ H $ be a normal subgroup of $ \pi_{1}(X, x_{0}) $. If the operation of taking inverse in $ (p^{-1}_{H}(x_{0}))^{wh} $ is continuous, then $ (p^{-1}_{H}(x_{0}))^{wh} $ is a topological group.
\end{proposition}

\begin{proof}
It suffices to show that the concatenation operation
$$\psi:(p^{-1}_{H}(x_{0}))^{wh}\times (p^{-1}_{H}(x_{0}))^{wh}\rightarrow (p^{-1}_{H}(x_{0}))^{wh}$$
defined by $\psi([\alpha]_{H},[\beta]_{H})=[\alpha\ast\beta]_{H}$ is continuous. Let $\alpha$ and $\beta$ be two loops at $x_{0}$ and $([\alpha\ast\beta]_{H},U)$ be an open neighborhood of the concatenation $[\alpha\ast\beta]_{H}$ in $ (p^{-1}_{H}(x_{0}))^{wh} $. The continuity of taking inverse, applied to the element $[\beta^{-1}]_{H}\in{(p^{-1}_{H}(x_{0}))^{wh}}$, implies the existence of a neighborhood $V$ of $x_{0}$ such that for any loop $\gamma$ in $V$ the inverse of $[\beta^{-1}\ast\gamma^{-1}]_{H}$ belongs to $([\beta]_{H},U)$, i.e, $([\beta^{-1}\ast\gamma^{-1}]_{H})^{-1}=[\gamma\ast\beta]_{H}=[\beta\ast\mu]_{H}$ for some loop $\mu$ in $U$. We consider open subsets  $([\alpha]_{H},V)$ and  $([\beta]_{H},U)$ containing  $[\alpha]_{H}$ and  $[\beta]_{H}$, respectively. It is enough to show that $\psi(([\alpha]_{H},V),([\beta]_{H},U)) \subseteq ([\alpha\ast\beta]_{H},U)$. Note that for any $[\alpha\ast\gamma]_{H}\in{([\alpha]_{H},V)}$ and any $[\beta\ast\delta]_{H}\in{([\beta]_{H},U)}$ the equality  $[\alpha\ast\gamma\ast\beta\ast\delta]_{H}=[\alpha\ast\gamma\ast\beta\ast\mu^{-1}\ast\beta^{-1}\ast\beta\ast\mu\ast\delta]_{H}
$ holds for some loop $\mu$ in $U$. Since  $[\gamma\ast\beta\ast\mu^{-1}\ast\beta^{-1}]\in{H}$, we have $[\alpha^{-1}\ast\alpha\ast\gamma\ast\beta\ast\delta\ast\delta^{-1}\ast\mu^{-1}\ast\beta^{-1}\ast
\alpha^{-1}\ast\alpha]\in{H}$. Also, since $H\unlhd \pi_{1}(X,x_{0})$, we have
$[\alpha\ast\gamma\ast\beta\ast\delta\ast\delta^{-1}\ast\mu^{-1}\ast\beta^{-1}\ast
\alpha^{-1}]\in{H}$ which implies that
$[\alpha\ast\gamma\ast\beta\ast\delta]_{H}=[\alpha\ast\beta\ast\mu\ast\delta]_{H}$. Therefore $[\alpha\ast\gamma\ast\beta\ast\delta]_{H}\in{([\alpha\ast\beta]_{H},U)}$ and hence the concatenation operation is continuous.
\end{proof}


Small loop transfer spaces appeared for the first time in \cite{BroU}.  In this section, we introduce a relative version of this notion and show that this new notion will be used to give an answer to the converse of Proposition 1.3.

\begin{definition}
Let $H\leq \pi_{1}(X,x_{0})$, then we say that the topological space $X$ is an $H$-SLT space at $x_{0}$, if for every path $\alpha$ begining at $x_{0}$ and for every open subset $x_{0}\in{U}$  there is an open subset $\alpha(1)\in{V}$ such that for every loop $\beta$ in $V$ based at $\alpha(1)$, there is a loop $\gamma$ in $U$ based at $x_{0}$ such that  $[\alpha\ast\beta\ast\alpha^{-1}]_{H}=[\gamma]_{H}$.
\end{definition}

It is easy to see that every SLT at $x_{0}$ space is an $H$-SLT at $x_{0}$, where $H$ is the trivial subgroup. Note that the converse does not hold in general (see Example 3.10).

Using the technique of the proof of Proposition 2.1, we can prove the following corollary. In fact, the property of being $H$-SLT at $x_{0}$ will be used in the proof of the continuity of the inverse map.


\begin{corollary}
Let $ H $ be a normal subgroup of $ \pi_{1}(X,x_{0}) $ and $ X $ be an $H$-SLT at $x_{0}$. Then $ (p^{-1}_{H}(x_{0}))^{wh} $ is a topological group.
\end{corollary}


Note that $p_{H}^{-1}(x_{0})$ is the set of right cosets of $H$ in $\pi_{1}(X,x_{0})$ (see \cite{Bcovering,H}). It is easy to check that $(p^{-1}_{H}(x_{0}))^{wh}$ and $(p^{-1}_{H}(x_{0}))^{top}$ are homeomorphic to the quotient spaces $ \frac{\pi^{wh}_{1}(X, x_{0})}{H} $ and $ \frac{\pi^{qtop}_{1}(X, x_{0})}{H}$, respectively. Therefore, as a consequence of Corollary 2.3, if $ H\unlhd\pi_{1}(X,x_{0}) $ and $ X $ is an $H$-SLT at $x_{0}$, then $ \frac{\pi^{wh}_{1}(X, x_{0})}{H} $ is a topological group. Moreover, this result hold for $ \frac{\pi^{qtop}_{1}(X, x_{0})}{H}$  if $X$ is locally path connected (see Corollary 3.2).  As a special case, if $X$ is an SLT space at $x_{0}$, then $\pi_{1}^{wh}(X,x_{0})$ is a topological group (see \cite[Theorem 4.12]{BroU}).

 Let $H\leq \pi_{1}(X,x_{0})$. Recall that $X$ is said to be homotopically Hausdorff relative to $ H $ if for every $x\in{X}$, for every $ \alpha \in P(X,x_0) $ with $\alpha(1)=x$, and for any $ g \notin H$, there is an open neighborhood $ U_g $ of $ \alpha(1) $ in which there is no loop $ \gamma:(I, \dot{I})\rightarrow (U_g,\alpha(1)) $  such that $ [ \alpha\ast\gamma\ast{\alpha}^{-1}]\in Hg (see $ \cite[p. 190]{Zastrow}). Note that $X$ is homotopically Hausdorff if and only if X is homotopically Hausdorff relative to the trivial subgroup $H = \lbrace 1\rbrace$. Also, $X$ is said to be homotopically path Hausdorff relative to $H$, if for every pair of paths $ \alpha, \beta \in P(X,x_0) $  with $ \alpha(1)=\beta(1) $  and $ [\alpha\ast{\beta}^{-1}] \notin H $, there is a partition $ 0=t_{0}<t_{1}<...<t_{n}=1 $ and a sequence of open subsets $ U_{1}, U_{2}, ..., U_{n} $ with  $ \alpha([t_{i-1},t_{i}])\subseteq U_{i} $ such that if $ \gamma:I \rightarrow X $  is another path satisfying $  \gamma([t_{i-1},t_{i}])\subseteq U_{i} $ for $1\leq i \leq n$ and $ \gamma(t_{i})=\alpha(t_{i})  $ for every $0\leq i\leq n $, then $ [\gamma \ast{\beta}^{-1}] \notin H (see $ \cite{BrazFa}).

Note that homotopically path Hausdorff relative to $H$ implies homotopically Hausdorff relative to $H$, but the converse does not hold in general \cite{Fischerexam, VZexam}. In the following theorem which is one of our main results, we show that the converse holds in SLT spaces relative to normal subgroups.


\begin{theorem}
Let $H\unlhd \pi_{1}(X,x_{0})$ and $X$ be an $H$-SLT at $x_{0}$. Then $X$ is homotopically path Hausdorff relative to $H$ if and only if $X$ is homotopically Hausdorff relative to $H$.
\end{theorem}

\begin{proof}
It suffices to show if $X$ is homotopically Hausdorff relative to $H$, then $X$ is homotopically path Hausdorff relative to $H$. Let $\alpha$ and $\beta$ be two paths in $X$ from $x_{0}$ to $x$  such that $[\beta\ast\alpha^{-1}]\notin{H}$ and $U$ be an open neighborhood of $x_{0}$ in $X$. For each $t\in{[0,1]}$, let $\alpha_{t}$ be the path from $x_{0}$ to $x_{t}=\alpha(t)$ determined by $\alpha$. By $H$-SLT property, for such $t\in{[0,1]}$, there is an open neighborhood $V_{t}$ of $x_{t}$ such that for any loop $\beta$ in $V_{t}$ at $x_{t}$ there is a loop $\lambda_{t}$ in $U$ at $x_{0}$ so that $[\alpha_{t}\ast\beta\ast\alpha^{-1}_{t}]_{H}=[\lambda_{t}]_{H}$.

Since $\alpha(I)$ is a compact subset of $X$, there is a partition $0=t_{0}\leq t_{1}\leq ...\leq t_{n}=1$ and   open neighborhoods $V_{1}, V_{2},..., V_{n}$ with $\alpha([t_{i-1},t_{i}])\subseteq V_{i} $. For simplicity, we put $\alpha|_{[t_{i-1},t_{i}]}=\delta_{i}$. Note that $\alpha_{t_{0}}=\delta_{0}=c_{x_{0}}$ and $\alpha_{t_{1}}=\delta_{1}$. Let $\gamma$ be another path such that $Im(\gamma_{i}=\gamma|_{[t_{i-1},t_{i}]})\subseteq V_{i} $ for $1\leq i\leq n$ and $\gamma(t_{i})=\alpha(t_{i})$ for every $0\leq i\leq n$. Using the loops $\lambda_{1}$, $\lambda_{2}$,...,$\lambda_{n}$ in $U$ introduced at the first paragraph of the proof we have
\begin{center}
 $[\alpha_{t_{i-1}}\ast\gamma_{i}\ast\delta^{-1}_{i}\ast\alpha^{-1}_{t_{i-1}}]_{H}=[\lambda_{i}]_{H}$,
\end{center}
or equivalently,
\begin{center}
$\theta_{i}=[\alpha_{t_{i-1}}\ast\gamma_{i}\ast\delta^{-1}_{i}\ast\alpha^{-1}_{t_{i-1}}]\in{H[\lambda_{i} ]}$,
\end{center}
for $1\leq i\leq n$.

Using normality of $H$, it is easy to show that
$ [\gamma\ast\alpha^{-1}]=\theta_{1}\theta_{2}...\theta_{n}\in H[\lambda_{1}\lambda_{2}...\lambda_{n}]$. If we put $\lambda=\lambda_{1}\lambda_{2}...\lambda_{n}$, then $[\gamma\ast\alpha^{-1}\ast\lambda^{-1}]\in{H}$.
On the other hand, since $X$ is homotopically Hausdorff relative to $H$, there is an open neighborhood $W$ of $x_{0}$ in $X$ such that for every $[\sigma]\in i_{\ast}\pi_{1}(W,x_{0})$, $[\sigma]\notin{H[\beta\ast\alpha^{-1}]}$ and hence  $[\sigma\ast\alpha\ast\beta^{-1}]\notin{H}$. If we consider $U=W$ and $\sigma=\lambda$, then we have $[\lambda\ast\alpha\ast\beta^{-1}]\notin{H}$. Since $[\gamma\ast\alpha^{-1}\ast\lambda^{-1}]\in{H}$, we have $[\gamma\ast\alpha^{-1}\ast\lambda^{-1}][\lambda\ast\alpha\ast\beta^{-1}]\notin{H}$. Thus  $[\gamma\ast\beta^{-1}]\notin{H}$ which implies that $X$ is homotopically path Hausdorff relative to $H$.
\end{proof}



The authors in \cite[Proposition 3.9]{Ab} showed that if $X$ is homotopically Hausdorff relative to $ H $, then $ H $ is closed in $ {\pi}^{wh}_{1}(X,x_{0}) $, but the converse does not hold. Consider \textit{Harmonic Archipelago} (denoted by HA) which is introduced by Virk \cite{Virk}. The space HA is semilocally simply connected at any point except the common point of the boundary circles. By \cite[Proposition 4.21]{BroU} $\pi_{1}^{wh}(X,x)$ is discrete if $X$ is semilocally simply connected at $x$. It turns out that the trivial subgroup of the fundamental group of HA is closed but the space HA is not homotopically Hausdorff because it has small loop (see \cite{Virk}).

In the following theorem, we give an  an answer to the converse of the above statement.


\begin{theorem}
Let $X$ be an $H$-SLT at $x_{0}$, where $H$ is a subgroup of $\pi_{1}(X,x_{0})$. Then $X$ is homotopically Hausdorff relative to $H$ if and only if $H$ is a closed subgroup of $\pi_{1}^{wh}(X,x_{0})$.
\end{theorem}

\begin{proof}
Let $\alpha\in{P(X,x_{0})}$ and $[\delta]\notin{H}$. Since $H$ is a closed subgroup of  $\pi_{1}^{wh}(X,x_{0})$, there is an open neighborhood $x_{0}\in{U}$ in $X$ such that $([\delta],U)\cap H=\emptyset$, i.e., for every loop $\beta$ in $U$ based at $x_{0}$ we have $[\delta\ast\beta]\notin{H}$. Since $X$ is $H$-SLT at $x_{0}$ so there is an open neighborhood $\alpha(1)\in{V}$ such that for every loop $\gamma$ in $V$ based at $\alpha(1)$ there is a loop $\beta$ in $U$ based at $x_{0}$  such that $[\alpha\ast\gamma\ast\alpha^{-1}]_{H}=[\beta]_{H}$, i.e., $[\alpha\ast\gamma\ast\alpha^{-1}]\in{H[\beta]}$. It suffices to show that for every loop $\gamma$ in $V$ based at $\alpha(1)$, we have $[\alpha\ast\gamma\ast\alpha^{-1}]\notin{H[\delta]}$ . Suppose
$[\alpha\ast\gamma\ast\alpha^{-1}]\in{[H\delta}]$ for some loop $\gamma$ in $U$. Then we have $[\alpha\ast\gamma\ast\alpha^{-1}]=[h_{1}\ast\delta]$ and  $[\alpha\ast\gamma\ast\alpha^{-1}]=[ h_{2}\ast\beta]$, where $h_{1}$ and $h_{2}$ are loops in X based at $x_{0}$ such that $[h_{1}]$ and $[h_{2}]$ belong to $H$. Consequently, we have $[ h_{1}\ast\delta]=[h_{2}\ast\beta]$, i.e., $[\delta\ast\beta^{-1}]\in{H}$ which is a contradiction.
\end{proof}

We can now determine characterization of $T_{0}$ axiom in $\pi_{1}^{wh}(X,x_{0})$ which follows directly from Theorem 2.5.
\begin{corollary}
Let $X$ be an SLT space. Then $X$ is homotopically Hausdorff if and only if $\pi_{1}^{wh}(X,x_{0})$ is $T_{0}$.
\end{corollary}


\begin{corollary}
Let  $H\unlhd{\pi}_{1}(X,x_{0}) $ and $X$ be a locally path connected $H$-SLT at $x_{0}$. Then $H$ is a closed subgroup of $ {\pi}^{qtop}_{1}(X,x_{0}) $ if and only if $H$ is a closed subgroup of $ {\pi}^{wh}_{1}(X,x_{0}) $.
\end{corollary}

\begin{proof}
We know that $ {\pi}^{wh}_{1}(X,x_{0}) $ is finer than $ {\pi}^{qtop}_{1}(X,x_{0}) $ \cite{VZcom}. Let $H$ be a closed subgroup of $ {\pi}^{wh}_{1}(X,x_{0}) $, by assumption and according to Theorem 2.4 and 2.5, $H$ is a closed subgroup of $ {\pi}^{qtop}_{1}(X,x_{0}) $.
\end{proof}

Some people determined sufficient conditions  to conclude that $p_{H}:\tilde{X}_{H}\rightarrow X$ has the unique path lifting property. For instance, Fisher and Zastrow \cite{FischerC} showed that if $H$ is an open subgroup of $\pi_{1}^{qtop}(X,x_{0})$, then $p_{H}$ is a semicovering and hence it has the unique path lifting property.  Fischer et al.
\cite[Theorem 2.9]{Fischerexam} proved that if $X$ is homotopically path Hausdorff, then $p_{H}$ has the unique path lifting property, when $H=1$. Finally, Brazas and Fabel \cite{BrazFa} generalized all the above cases of unique path lifting and showed that if $X$ is locally path connected and $H$ is a closed subgroup of $\pi_{1}^{qtop}(X,x_{0})$, then $p_{H}$ has the unique path lifting property. Zastrow has announced that the converse of the above statement does not hold for Peano continua in the case $H = 1$ (see \cite{VZexam}). Note that the locally path connected space $X$ is homotopically path Hausdorff relative to $C$ if and only if $C$ is a closed subset in $\pi_{1}^{qtop}(X,x_{0})$ \cite[Lemma 9]{BrazFa}. Combining Theorem 2.5 and Corollary 2.7, we obtain the following corollary which answers the converse of the above result or equivalently the converse of part (ii) of Proposition 1.3.

\begin{corollary}
 Let $X$ be a locally path connected, SLT at $x_{0}$ and $H\unlhd \pi_{1}(X,x_{0})$. Then $p_{H}$ has the unique path lifting property if and only if $H$ is a closed subgroup of $ {\pi}^{qtop}_{1}(X,x_{0}) $.
\end{corollary}

\begin{proof}
Let $p_{H}$ has the unique path lifting, then according to \cite[Proposition 6.4]{Zastrow} $X$ is homotopically Hausdorff relative to $H$. Using Theorem 2.5, since $X$ is an SLT space at $x_{0}$, $H$ is closed a subgroup of $\pi_{1}^{wh}(X,x_{0})$. Therefore, by Corollary 2.7, $H$ is a closed subgroup of $ {\pi}^{qtop}_{1}(X,x_{0})$.
\end{proof}

\begin{corollary}
Let $X$ be locally path connected, $H$-SLT at $x_{0}$ and homotopically Hausdorff relative to $ K $, where $K$ is a normal subgroup of $\pi_{1}(X,x_{0})$ containing $H$. Then $p_{K}:\tilde{X}_{K}\rightarrow X$ has the unique path lifting property.
\end{corollary}

\begin{proof}
Since $X$ is $H$-SLT at $x_{0}$ and $H\leq K\unlhd \pi_{1}(X,x_{0})$, $X$ is a $K$-SLT space at  $x_{0}$. By assumption and Theorem 2.4, $X$ is homotopically path Hausdorff relative to $K$. Therefore, using \cite[Lemma 9]{BrazFa} $K$ is a closed subgroup of $ {\pi}^{qtop}_{1}(X,x_{0}) $ and hence $p_{K}:\tilde{X}_{K}\rightarrow X$ has the unique path lifting property.
\end{proof}
Small loop transfer spaces have remarkable relation with the compact-open topology and the whisker topology on $\tilde{X}_{e}$.  For example, Brodskiy et al. in \cite{BroU} showed that SLT spaces are equivalent with  $\tilde{X}_{e}^{wh}=\tilde{X}_{e}^{top}$. Similarly, it is of interest to determine when the compact-open topology coincides with the whisker topology on $ \tilde{X}_H $. In other words,
\begin{center}
`` What is the equivalent condition for $\tilde{X}_{H}^{wh}=\tilde{X}_{H}^{top}$? ''
\end{center}

We will answer to this question in the next section. To answer this question, we need to introduce the following notion.

\begin{definition}
Let $ H $ be any subgroup of $ \pi_{1}(X,x_{0}) $. Then $X$ is called an $H$-SLT space if for every $x\in{X}$ and for every path $\lambda$ from $x_{0}$ to $x$, $X$ is a $[\lambda^{-1}H\lambda]$-SLT at $x$.
\end{definition}
\begin{remark}
Recall that $X$ is SLT space if it is SLT at $x$ for every $x\in{X}$. Note that this fact is not correct for $H$-SLT spaces because $H$ is not a subgroup of  $ \pi_{1}(X,x) $ for any point $x\neq x_{0}$. To present a similar fact, we can consider subgroups corresponding to $H$ in $ \pi_{1}(X,x) $ by the isomorphism $\varphi_{\lambda}:\pi_{1}(X,x_{0})\rightarrow \pi_{1}(X,x)$ for every path $\lambda$ from $x_{0}$ to $x$. We denote $[\lambda]^{-1}H[\lambda]$  by $[\lambda^{-1} H \lambda]$.
\end{remark}

Brodskiy et al. \cite{BroU} defined the map $h_{\alpha}:\pi_{1}^{wh}(X,y)\rightarrow \pi_{1}^{wh} (X,x)$ by $h_{\alpha}([\beta])=[\alpha\ast\beta\ast\alpha^{-1}]$, where $\alpha$ is a path from $x$ to $y$ and showed that the path $\alpha$ has the property SLT if and only if $h_{\alpha}$ is continuous. We can explain this result by the fibers of the endpoint projection $p_{e}:\tilde{X}_{e}\rightarrow X$ as follows. The path $\alpha$ has the property SLT if and only if the map $f^{\alpha}_{e}:(p^{-1}_{e}(y))^{wh}\rightarrow (p_{e}^{-1}(x))^{wh}$  defined by $f^{\alpha}_{e}([\delta])=[\delta\ast\alpha^{-1}]$ is continuous. Moreover, we can describe $H-$SLT spaces in terms of the map $f^{\alpha}_{H}: (p^{-1}_{H}(y))^{wh}\rightarrow (p_{H}^{-1}(x))^{wh}$ where  $f^{\alpha}_{H}([\delta]_{H})=[\delta\ast\alpha^{-1}]_{H}$.

\begin{proposition}
Let $ H\leq \pi_{1}(X,x_{0}) $ and $ p_{H}:\tilde{X}_{H}\rightarrow X $ be the endpoint projection. Then $ X $ is an $H$-SLT space if and only if $ f^{\alpha}_{H}:(p^{-1}_{H}(y))^{wh}\rightarrow (p^{-1}_{H}(x))^{wh} $ is continuous for every path $ \alpha $ whit $\alpha(0)=x$, $\alpha(1)=y$.
\end{proposition}

\begin{proof}
 The map $f^{\alpha}_{H}$ is clearly bijection. It suffices to show that $f^{\alpha}_{H}$ is continuous. Pick  $[\delta]_{H}\in{p_{H}^{-1}(y)}$. Let  $N=([ \delta\ast \alpha^{-1}]_{H},U)$ be an open basis in $(p_{H}^{-1}(x))^{wh}$. Since $X$ is an $H$-SLT space, there is an open subset $V$ of $X$ based at $y$ such that for every loop $\beta$ in $V$ based at $y$ there is a loop $\gamma$ in $U$ based at $x$ such that $[\alpha\ast\beta\ast\alpha^{-1}]_{[\alpha\ast\delta^{-1}H\delta\ast\alpha^{-1}]}=[\gamma]_{[\alpha\ast\delta^{-1}H\delta\ast\alpha^{-1}]}$, or equivalently $[\delta\ast\beta\ast\alpha^{-1}]_{H}=[\delta\ast\alpha^{-1}\ast\gamma]_{H}$.
Consider open basis $M=([\delta]_{H},V)$ of $p_{H}^{-1}(y)$. We show that $f^{\alpha}_{H}(M)\subseteq N$. According to the definition of $f^{\alpha}_{H}$ we know that $f^{\alpha}_{H}([\delta\ast\beta]_{H})=[\delta\ast\beta\ast\alpha^{-1}]_{H}$. On the other words, we have $[\delta\ast\beta\ast\alpha^{-1}]_{H}=[\delta\ast\alpha^{-1}\ast\gamma]_{H}$ which implies that $f^{\alpha}_{H}$ is continuous.

Conversely, pick $x\in{X}$. Let $\lambda:I\rightarrow X$ be a path from $x_{0}$ to $x$, and $U$ be an open neighbourhood of $x$. By assumption, the map $ f^{\alpha}_{H}:(p^{-1}_{H}(y))^{wh}\rightarrow (p^{-1}_{H}(x))^{wh} $ is continuous, where $f^{\alpha}_{H}([\lambda\ast\alpha]_{H})=[\lambda]_{H}$. Therefore there  is an open subset $y\in{V}$ such that $f^{\alpha}_{H}([\lambda\ast\alpha]_{H},V))\subseteq ([\lambda]_{H},U))$. Thus, for every loop $\beta$ in $V$ there is a loop $\gamma$ in $U$ such that $[\lambda\ast\alpha\ast\beta\ast\alpha^{-1}]_{H}=[\lambda\ast\gamma]_{H}$ or equivalently $[\alpha\ast\beta\ast\alpha^{-1}]_{[\lambda^{-1}H\lambda]}=[\gamma]_{[\lambda^{-1}H\lambda]}$ which implies that $X$ is $H$-SLT space.
\end{proof}


\begin{definition}
A path $\alpha:I\rightarrow X$ from $x$ to $y$ is called an $H$-SLT path when the map $f^{\alpha}_{H}:(p^{-1}_{H}(y))^{wh}\rightarrow (p^{-1}_{H}(x))^{wh} $ defined by $ f_{\alpha}([\delta]_{H})=[ \delta\ast \alpha^{-1}]_{H}$ is continuous, or equivalently, for every path $\lambda$ from $x_{0}$ to $x$ and for every open neighborhood U of $\alpha(0)$ in $X$, there exists an open neighborhood $V$ of $\alpha(1)$ in $X$ such that for every loop $\beta:(I,0)\rightarrow (V,\alpha(1))$ there is a loop $\gamma:(I,0)\rightarrow (U,\alpha(0))$ such that $ [\alpha\ast\beta\ast\alpha^{-1}]_{[\lambda^{-1}H\lambda]}=[\gamma]_{[\lambda^{-1}H\lambda]} $. Note that $X$ is an $H$-SLT space if every path in $X$ is an $H$-SLT path.
\end{definition}


Semicovering maps are introduced by Brazas \citep{BrazS}. A semicovering map $p:\tilde{X}\rightarrow X$ is a local homeomorphism with continuous lifting of paths and homotopies \citep{BrazO}. Recall that the fibers of covering and semicovering map are homeomorphic \citep{BrazO}. In comparison with $\textbf{lpc}_{0}$-covering spaces, there is a generalized universal covering space with non-homeomorphic fibers \citep[Example 4.15]{Zastrow}. Every $\textbf{lpc}_{0}$-covering map $p:\tilde{X}\rightarrow X$ is equivalent to  $p_{H}:\tilde{X}_{H}\rightarrow X$, where $H=p_{\ast}\pi_{1}(\tilde{X},\tilde{x}_{0})$ \citep[Lemma 5.10]{BrazG}. Using Proposition 2.12 we obtain the following corollary for $\textbf{lpc}_{0}$-covering maps.


\begin{corollary}
If $ X $ is a path connected SLT-space, then all fibers of $\textbf{lpc}_{0}$-covering map $p:\tilde{X}\rightarrow X$ are homeomorphic.
\end{corollary}

Using the property of the unique path lifting property of semicovering map, we can clearly show  that if $H$ is a subgroup of $\pi_{1}(X,x_{0})$ then $\tilde{H}$ is a  subgroup of $\pi_{1}(\tilde{X},\tilde{x_{0}})$, where $p_{\ast}\tilde{H}=H$ and $p(\tilde{x_{0}})=x_{0}$.


\begin{proposition}
Let $p:\tilde{X}\rightarrow X$ be a semicovering map and let $\alpha$ be an $H$-SLT path from $x_{0}$ to $x$ in $X$. Then the lift of $\alpha$ in $\tilde{X}$ is an $\tilde{H}$-SLT path in $\tilde{X}$.
\end{proposition}

\begin{proof}
Let $\alpha$ be an $H$-SLT path from $x_{0}$ to $x$ and $\tilde{\alpha}$ be the lift of $\alpha$ with $\tilde{\alpha}(0)=\tilde{x}_{0}$ and $\tilde{\alpha}(1)=\tilde{x}$. Without loss of generality, let $\tilde{U}$ be an open neighborhood in $\tilde{X}$ containing of $\tilde{x}_{0}$ such that $p\vert_{\tilde{U}}:\tilde{U}\rightarrow U$ is a homeomorphism. Now, since $\alpha$ is an $H$-SLT path, there is an open neighborhood $V$ of $\alpha(1)$ such that for every loop $\beta$ in $V$ at $\alpha(1)$ there is a loop $\delta$ in $U$ at $\alpha(0)$ such that $[\alpha\ast \beta\ast\alpha^{-1}\ast\delta^{-1}]\in{H}$. Again without loss of generality, let $\tilde{W}$ be an open neighborhood of $\tilde{X}$ such that $p\vert_{\tilde{W}}:\tilde{W}\rightarrow W$ is a homeomorhism, where $W$ is an open subset of $V$. By the uniqueness of path lifting property of the map $p:\tilde{X}\rightarrow X$ we have $[\tilde{\alpha}\ast \tilde{\beta}\ast\tilde{\alpha}^{-1}\ast\tilde{\delta}^{-1}]\in{\tilde{H}}$. Therefore we imply that $\tilde{\alpha}$ is an $ \tilde{H} $-SLT path.
\end{proof}


\begin{corollary}
Let $p:\tilde{X}\rightarrow X$ be a semicovering map. If $X$ is an SLT space, then so is $\tilde{X}$.
\end{corollary}






\section{Relationship Between Open Subsets of $\pi_{1}^{wh}$ and $\pi_{1}^{qtop}$}

Brodskiy et al. \cite[Theorem 4.11]{BroU} showed that the coincidence of the compact-open topology and the whisker topology on $\tilde{X}_{e}$ provides a property on $X$ that they call it small loop transfer, i.e, the space $X$ is an SLT space. Also, they proved that the converse holds for Peano spaces (see \cite[Theorem 4.12]{BroU}). In the following we extend these results to $H$-SLT spaces.

\begin{theorem}
Let $ H $ be a normal subgroup of $ \pi_{1}(X,x_{0}) $ and $ X $ be a locally path connected space. Then $X$ is an $H$-SLT space if and only if $ \tilde{X}_{[\lambda^{-1} H\lambda]}^{wh}=\tilde{X}_{[\lambda^{-1} H\lambda]}^{top}$ for every path $\lambda$ from $x_{0}$ to any $x$ in $X$.
\end{theorem}

\begin{proof}





In general, the whisker topology on $ \tilde{X}^{wh}_{H} $ is finer than the compact-open topology on $ \tilde{X}^{top}_{H} $ (see \cite[Remark 6.1]{BrazG}). Hence the whisker topology on $ \tilde{X}^{wh}_{[\lambda^{-1} H\lambda]} $  is finer than the compact-open topology on $ \tilde{X}^{top}_{[\lambda^{-1} H\lambda]} $ for every path $\lambda$ from $x_{0}$ to any $x$. First, we show that $ \tilde{X}^{top}_{H} $ is finer than $ \tilde{X}^{wh}_{H} $ when $X$ is an $H$-SLT space and $H\unlhd \pi_{1}(X,x_{0})$. Let $x\in{X}$ and $\alpha$ be a path in $X$ from $x_{0}$ to $x$. Let $ ([\alpha]_{H},U) $ be a basis open set of $ \tilde{X}_{H}^{wh}$. We show that there exists an open basis $W=\bigcap_{i=1}^{n} \langle k_{i},U_{i}\rangle$ of $P(X,x_{0})$ which contains $\alpha$ such that $q_{H}(W)\subset ([\alpha]_{H},U)$, where $q_{H}:P(X,x_{0})\rightarrow \tilde{X}_{H}$ is the quotient map. For each $t\in{[0,1]}$, let $\alpha_{t}$ be the path from $x_{t}=\alpha(t)$ to $x$ determined by $\alpha$. For each $t\in{[0,1]}$ choose an open path connected neighborhood $V_{t}$ of $x_{t}$ such that for any loop $\beta$ in $V_{t}$ at $x_{t}$ there is a loop $\delta$ in $U$ at $x$ so that $[\alpha_{t}^{-1}\ast\beta\ast\alpha_{t}]_{[\alpha^{-1} H\alpha]}=[\delta]_{[\alpha^{-1} H\alpha]}$.

For each $t\in{I}$ choose a closed subinterval $I_{t}$ of $I$ containing $t$ in its interior so that $\alpha(I_{t})\subset V_{t}$. Choose finitely many of them that cover $I$ so that no proper subfamily cover $I$. Let $S$ be the set of such chosen points $t$. If $I_{s}\cap I_{t}\neq \emptyset$, let $V_{s,t}$ be the path component of $V_{s}\cap V_{t}$ containing $\alpha(I_{s}\cap I_{t})$. Otherwise, $V_{s,t}=\emptyset$. Let $W$ be the set of all paths $\beta$ in $X$ beginning at $x_{0}$ such that $\beta(I_{s}\cap I_{t})\subset V_{s,t}$ for all $s,t\in{S}$. Given $\beta$ in $W$ pick points $x_{s,t}\in{I_{s}\cap I_{t}}$ if $s\neq t$ and $I_{s}\cap I_{t}\neq\emptyset$. Then choose paths $\gamma_{s,t}$ in $X$ from $\alpha(x_{s,t})$ to $\beta(x_{s,t})$. Arrange points $x_{s,t}$ in an increasing sequence $y_{i}$, $1\leqslant i\leqslant n$, put $y_{0}=0$ and $y_{n+1}=1$, and create loops $\lambda_{i}$ in $X$ at $\alpha(y_{i})$ as follows: travel along $\alpha$ from $\alpha(y_{i})$ to $\alpha(y_{i+1})$, then along $\gamma_{i+1}$, then reverse $\beta$ from $\beta(y_{i+1})$ to $\beta(y_{i})$, finally reverse $\gamma_{i}$. Notice $\alpha^{-1}\ast\beta \simeq \prod_{i=1}^{n} \alpha_{i}^{-1}\ast\lambda_{i}\ast\alpha_{i}$, where $\alpha_{i}$ is the path determined by $\alpha$ from $\alpha(y_{i})$ to $x$. Put $K=[\alpha^{-1}H\alpha]$. Since $X$ is an $H$-SLT space, $[ \alpha_{i}^{-1}\ast\lambda_{i}\ast\alpha_{i}]_{K}=[\delta_{i}]_{K}$ i.e. $[\alpha_{i}^{-1}\ast\lambda_{i}\ast\alpha_{i}]\in{K[\delta_{i}]}$ for some loop $\delta_{i}$ at $x$ in $U$, where $i=1,...,n$. By assumption, since $H$ is a normal subgroup of $\pi_{1}(X,x_{0})$, $[\prod_{i=1}^{n} \alpha_{i}^{-1}\ast\lambda_{i}\ast\alpha_{i}]_{K}=[\alpha_{1}^{-1}\ast\lambda_{1}\ast\alpha_{1}\ast\alpha_{2}^{-1}\ast\lambda_{2}\ast\alpha_{2}\ast..........\ast\alpha_{n}^{-1}\ast\lambda_{n}\ast\alpha_{n}]\in{K[\delta]}$, where $\delta=\delta_{1}\ast\delta_{2}\ast...\ast\delta_{n}$. Since $[\alpha^{-1}\ast\beta]_{K}=[\prod_{i=1}^{n} \alpha_{i}^{-1}\ast\lambda_{i}\ast\alpha_{i}]_{K}$, we have $[\alpha^{-1}\ast\beta]_{K}=[\delta]_{K}$, $[\alpha^{-1}\ast\beta\ast\delta^{-1}]\in{[\alpha^{-1}H\alpha]}$ and $[\beta\ast\delta^{-1}\ast\alpha^{-1}]\in{H}$. Therefore $ [\beta]_{H}=[\alpha\ast\delta]_{H} $ which implies that $[\beta]_{H}\in{([\alpha]_{H},U)}$. Thus $ \tilde{X}^{top}_{H} $ is finer than $ \tilde{X}^{wh}_{H} $. Note that since $X$ is an $H$-SLT space, $X$ is a $[\lambda^{-1} H\lambda]$-SLT space for every path $\lambda$ from $x_{0}$ to any $x$ in $X$ where $[\lambda^{-1} H\lambda]$ is a normal subgroup of $ \pi_{1}(X,x_{0}) $. Hence $ \tilde{X}^{top}_{[\lambda^{-1} H\lambda]} $ is finer than $ \tilde{X}^{wh}_{[\lambda^{-1} H\lambda]} $.

Conversely, pick $x\in{X}$ and let $\delta:I\rightarrow X$ be a path from $x_{0}$ to $x$, $\sigma:I\rightarrow X$ be a path from $x$ to $y$ and $U$ be an open neighborhood in $X$ at $x$. By assumption, $ \tilde{X}^{wh}_{[\lambda^{-1}H\lambda]}=\tilde{X}^{top}_{[\lambda^{-1}H\lambda]} $ for every path $\delta\rightarrow X$ with $\delta(0)=x_{0}$ and $\delta(1)=x$.
Put $K=[\sigma^{-1}\ast\delta^{-1}H\delta\ast\sigma]$. We know that the quotient map $\pi:P(X,y)\rightarrow \tilde{X}^{wh}_{K}=\tilde{X}^{top}_{K}$ is continuous, typically $\pi(\sigma^{-1})=[\sigma^{-1}]_{K}$. Take the open basis $([\alpha^{-1}]_{K},U)$ of $\tilde{X}^{wh}_{K}$. By the continuity of the map $\pi$, $\pi^{-1}([\sigma^{-1}]_{K},U)$ is an open subset of $P(X,y)$. According to \cite[Lemma 4.3]{BroU}, there are open subsets $y\in{V}$, $x\in{W}$ and $([\sigma^{-1}]_{K},V,W)$ such that $([\sigma^{-1}]_{K},V,W)\subseteq ([\sigma^{-1}]_{K},U)$. Therefore, for every loop $\beta$ in $V$ based at $y$ there is a loop $\gamma$ in $U$ based at $x$ such that $[\sigma\ast\beta\ast\sigma^{-1}]_{[\delta^{-1}H\delta]}=[\gamma]_{[\delta^{-1}H\delta]}$ which implies that $X$ is an $H-$SLT space.
\end{proof}
The following corollary is a direct consequence of Theorem 3.1.
\begin{corollary}
Let $ H $ be any normal subgroup of $ \pi_{1}(X,x_{0}) $ and let $ X $ be a locally path connected space. Then $X$ is an $H$-SLT at $x_{0}$ if and only if $ (p^{-1}_H(x_0))^{top}=(p^{-1}_H(x_0))^{wh} $, or equivalently, $\frac{\pi_{1}(X,x_{0})^{wh}}{H}=\frac{\pi_{1}(X,x_{0})^{qtop}}{H}$.
\end{corollary}
Fischer and Zastrow \cite[Lemma 2.1]{Zastrow} showed that the topology of $\tilde{X}^{wh}_{e}$ is finer than the topology of $\tilde{X}^{top}_{e}$. Since the fundamental group $\pi_{1}(X,x_{0})$ is equal to the subset of all classes of
loops in the set $\tilde{X}_{e}$, one can check that the proof of \cite[Lemma 2.1]{Zastrow} holds also for  $ \pi_{1}(X,x_{0}) $. In other words, we can conclude from the proof of \cite[Lemma 2.1]{Zastrow} that $\pi_{1}^{wh}(X,x)$ and $\pi_{1}^{qtop}(X,x)$ are agree when $X$ is semilocally simply connected. Note that Virk and Zastrow \cite{VZcom} illustrated that these two topologies does not agree, in general. Also, Brodskiy et al. in \cite[Theorem 4.12]{BroU} state the equivalent condition for this coincidence.

Using Corollary 3.2, we can conclude one of the main result of this section.

\begin{corollary}
Let $ H $ be any normal subgroup of $ \pi_{1}(X,x_{0}) $ and $ X $ be a locally path connected $H-$SLT space at $x_{0}$. Then any subset $U$ of $\pi_{1}(X,x_{0})$ containing $H$ is open in $\pi_{1}^{wh}(X,x_{0})$ if and only if it is open in $\pi_{1}^{qtop}(X,x_{0})$.
\end{corollary}




It is easy to see that the property of being small loop transfer for $X$ is a necessary condition for openness of the trivial subgroup $H=1$ in $\pi_{1}^{qtop}(X,x_{0})$. In the following proposition we extend this result to $H$-SLT spaces.

\begin{proposition}
Let $X$ be locally path connected and $H$ be an open subgroup of $\pi_{1}^{qtop}(X,x_{0})$, then $X$ is an $H$-SLT space.
\end{proposition}

\begin{proof}
Let $x\in{X}$ and $\delta:I\rightarrow X$ be a path from $x_{0}$ to $x$. Consider $\alpha$ a path from $x$ to $y$ and $U$ an open neighbourhood of $x$. Using \cite[Corollary 3.3]{TorabiS} and hypotheses, there is an open set $V$ of $y$ such that $[\delta\ast\alpha]V[\alpha^{-1}\ast\delta^{-1}]\leq H$, i.e., for every loop $\beta$ in $V$ based at $y$ we have $[\delta\ast\alpha\ast\beta\ast\alpha^{-1}\ast\delta^{-1}]\in{H}$ or equivalently $[\alpha\ast\beta\ast\alpha^{-1}]\in{[\delta^{-1}H\delta]}$. Hence $[ \alpha\ast\beta\ast\alpha^{-1}]_{[\delta^{-1}H\delta]}=[c_{x}]_{[\delta^{-1}H\delta]} $ which implies that $X$ is an $H-$SLT space.
\end{proof}

\begin{proposition}
Let $H\leq \pi_{1}(X,x_{0})$ and $X$ be a locally path connected $H-$SLT at $x_{0}$. Then $H$ is an open subgroup of $\pi_{1}^{wh}(X,x_{0})$ if and only if $H$ is an open subgroup of $\pi_{1}^{qtop}(X,x_{0})$.
\end{proposition}

\begin{proof}
Let $H$ be an open subgroup of $\pi_{1}^{wh}(X,x_{0})$. According to the definition of the whisker topology, there is an open basis $i_{\ast}\pi_{1}(U,x_{0})$ such that  $i_{\ast}\pi_{1}(U,x_{0}) \leq H$. On the other hand, since $X$ is an $H-$SLT at $x_{0}$, there is  a path open cover $\mathcal{V}=\lbrace V_{\alpha} \vert \alpha\in{P(X,x_{0})} \rbrace$ such that for any path $\alpha$ and for every loop $\beta$ in $V_{\alpha}$, there is a loop $\delta$ in $U$ based at $x_{0}$ such that $[\alpha\ast\beta\ast\alpha^{-1}]_{H}=[\delta]_{H}$ or equivalently, $[\alpha\ast\beta\ast\alpha^{-1}]\in{H[\delta]}$. Since $i_{\ast}\pi_{1}(U,x_{0})\leq H$, we have $[\alpha\ast\beta\ast\alpha^{-1}]\in{H}$. Thus we obtain a path Spanier group of the form $\tilde{\pi}(\mathcal{V},x_{0})$ such that $\tilde{\pi}(\mathcal{V},x_{0})\leq H$ (see \cite{TorabiS}). Therefore, Corollary 3.3 of \cite{TorabiS} implies that $H$ is an open subgroup of $\pi_{1}^{qtop}(X,x_{0})$.
\end{proof}



The following remark help us to determine the form of some semicovering subgroups in SLT spaces relative to subgroups .
\begin{remark}
Let $H\leq K\leq \pi_{1}(X,x_{0})$. It follows directly of Definition 2.2 that all $H$-SLT at $x_{0}$ spaces are $K-$SLT at $x_{0}$. Therefore, using Proposition 3.5, if $X$ is $H$-SLT at $x_{0}$, then $K$ is an open subgroup of $\pi_{1}^{wh}(X,x_{0})$ if and only if $K$ is an open subgroup of $\pi_{1}^{qtop}(X,x_{0})$ .
\end{remark}

Although the fibers in semicovering spaces are discrete but this result does not hold in $\textbf{lpc}_{0}$-covering spaces because the fibers are not necessarily homeomorphic. We verify that the fibers in $\textbf{lpc}_{0}$-covering spaces are Hausdorff. At first, in the following lemma, we show that the property of homotopically Hausdorff relative to $H$ has a significant influence to the fibers of the endpoint projection $p_{H}$.

\begin{lemma}
Let $X$ be path connected and $H$ be a subgroup of $\pi_{1}(X,x_{0})$. Then $X$ is homotopically Hausdorff relative to $H$ if and only if $(p_{H}^{-1}(x))^{wh}$ is Hausdorff for every $x\in{X}$.
\end{lemma}
\begin{proof}
Let $X$ be a homotopically Hausdorff relative to $H$ and pick $x\in{X}$. We show $p_{H}^{-1}(x)$ is Hausdorff. Consider $[\alpha]_{H}$ and $[\beta]_{H}$ belong to $p_{H}^{-1}(x)$, where $\alpha$ and $\beta$ are paths from $x_{0}$ to $x=\alpha(1)=\beta(1)$ such that $[\alpha\ast\beta^{-1}]\notin{H}$. Since $X$ is  homotopically Hausdorff relative to $H$, there is a neighborhood $U$ of $x$ in $X$ such that neither $[\alpha\ast\delta\ast\alpha^{-1}]\in H[\alpha\ast\beta^{-1}]$ nor $[\beta\ast\delta\ast\beta
^{-1}]\in H[\alpha\ast\beta^{-1}]$ for any loop $\delta$ based at $x$ in $U$. We show that $([\beta]_{H},U)\cap ([\alpha]_{H},U)=\emptyset$. By contrary, suppose $[\beta\ast\delta]_{H}=[\alpha\ast\lambda]_{H}$ for some loop $\delta$ and $\lambda$ in $U$ based at $x$. So, we have $[\beta\ast\delta\ast\lambda^{-1}\ast\alpha^{-1}]\in{H}$, or equivalently $[\beta\ast\delta\ast\lambda^{-1}\ast\beta^{-1}\ast\beta\ast\alpha^{-1}]\in{H}$ which implies that $[\beta\ast\delta\ast\lambda^{-1}\ast\beta^{-1}]\in H[\alpha\ast\beta^{-1}]$. This is a contradiction.

Conversely, pick $x\in{X}$ and let $p_{H}^{-1}(x)$ be Hausdorff. Let $[\lambda]\notin{H}$ and $\alpha$ be a path from $x_{0}$ to $x$. Since $[\lambda]\notin{H}$, $[\lambda\ast\alpha]_{H}\neq [\alpha]_{H}$. By assumption, there are neighborhoods $U$ and $V$ such that $([\lambda\ast\alpha]_{H},U)\cap ([\alpha]_{H},V)=\emptyset$. Put $W=U\cap V$. We can see that $[\alpha]_{H}\neq [\lambda\ast\alpha\ast\delta^{-1}]_{H}$ for every loop $\delta$ based at $x$ in $W$. Therefore, $[\alpha\ast\delta\ast\alpha^{-1}\ast\lambda^{-1}]\notin{H}$ and thus there is no loop  $\delta$ in $U$ such that $[\alpha\ast\delta\ast\alpha^{-1}]\in{H[\lambda]}$ which implies that $X$ is homotopically Hausdorff relative to $H$.
\end{proof}
Since every $\textbf{lpc}_{0}$-covering map is equivalent to a certain endpoint map, the following corollary will be implied from the above lemma.

\begin{corollary}
If $p:\tilde{X}\rightarrow X$ is $\textbf{lpc}_{0}$-covering map with $p_{\ast}\pi_{1}(\tilde{X},\tilde{x}_{0})=H\leq \pi_{1}(X,x_{0})$, then the fibers of $p$ are Hausdorff.
\end{corollary}

In the following corollary, we present some conditions which make $\textbf{lpc}_{0}$-covering map to be a semicovering map.

\begin{corollary}
Let $p:\tilde{X}\rightarrow X$ be a $\textbf{lpc}_{0}$-covering map with  $p_{\ast}\pi_{1}(\tilde{X},\tilde{x}_{0})=H\leq \pi_{1}(X,x_{0})$ and $X$ be $H$-SLT at $x_{0}$. If $\vert p^{-1}(x_{0}) \vert < \infty$, then $p$ is a semicovering map.
\end{corollary}

\begin{proof}
It is easy to see that every finite subset of a Hausdorff space is discrete. Therefore, using Corollary 3.8, the fiber of $p$ which is homeomorphism to $\frac{\pi_{1}(X,x_{0})^{wh}}{H}$ is discrete or equivalently, $H$ is an open subgroup of $\pi_{1}^{wh}(X,x_{0})$. On the other hand, using Proposition 3.5, $H$ is an open subgroup of $\pi_{1}^{qtop}(X,x_{0})$ which implies that $p$ is a semicovering map.
\end{proof}

It turns out that any semilocally simply connected space is a SLT space and any SLT space is an $H$-SLT space for every subgroup $H$ of $\pi_{1}(X,x_{0})$. Also, note that any space $X$ is a $\pi_{1}(X,x_{0})$-SLT space. In the following, we give an example of an $H$-SLT space which is not an SLT space, accordingly, it is not semilocally simply connected, where $H\neq \pi_{1}(X,x_{0})$.

\begin{example}
Let $(S^{1},0)$ be a unit circle, $(HA,x)$ be the Harmonic Archipelago, where $x$ is the common point of boundary circles. We consider the wedge space of $X=\frac{S^{1} \sqcup HA}{x\sim 0}$. In \cite[Example 4.4]{Torabi} it is shown that $\pi_{1}(X,x_{0})\neq \pi_{1}^{sg}(X,x_{0})$. On the other hand, $X$ is a semilocally small generated space \cite{Torabi}. Accordingly, $\pi_{1}^{sg}(X,x_{0})$, introduced by Virk \cite{Virk}, is an open subgroup of $\pi_{1}^{qtop}(X,x_{0})$. Using Proposition 3.4, we conclude that $X$ is a $\pi_{1}^{sg}(X,x_{0})$-SLT space. It is not hard to show that $X$ is not an SLT space. To prove that $X$ is not an SLT space, consider an arbitrary path (for example $\alpha$) in $X$ inside of $HA$ from any semilocally simply connected point to the wedge point.
\end{example}

\begin{remark}
Note that since the topology of $\pi_{1}^{\tau}(X,x)$ is coarser than the one of $\pi_{1}^{qtop}(X,x)$, in fact $\pi_{1}^{\tau}(X,x)$ and $\pi_{1}^{qtop}(X,x)$ have the same open subgroups \cite[Proposition 3.16]{BrazT}, it is routine to check that all the results of this section hold if we replace $\pi_{1}^{qtop}(X,x)$ with $\pi_{1}^{\tau}(X,x)$.
\end{remark}











\section*{Acknowledgments}
This research was supported by a grant from Ferdowsi University of Mashhad-Graduate Studies (No. 38529).

\section*{Reference}

\bibliography{mybibfile}

\begin{thebibliography}{99}

\bibitem{Ab}{} M. Abdullahi Rashid, B. Mashayekhy, H. Torabi, S.Z. Pashaei, On
topologized fundamental subgroups and generalized coverings, to appear in Bull. Iranian Math. Soc.

\bibitem{Beres}{} V. Berestovskii, C. Plaut, Uniform universal covers of uniform spaces, Topol. Appl. 154
(2007) 1748--1777.

\bibitem{Biss}{} D.K. Biss, The topological fundamental group and generalized covering spaces, Topol. Appl. 124 (2002) 355--371.

\bibitem{Bog}{} W.A. Bogley, A.J. Sieradski, Universal path spaces, preprint, http://oregonstate.edu/~bogleyw/.

\bibitem{BrazS}{} J. Brazas, Semicoverings: a generalization of covering space theory, Homol. Homotopy Appl. 14 (2012) 33--63.

\bibitem{BrazT}{} J. Brazas, The fundamental group as a topological group, Topol. Appl. 160 (2013) 170–-188.

\bibitem{BrazO}{} J. Brazas, Semicoverings, coverings, overlays and open subgroups of the quasitopological fundamental group, Topology Proc. 44 (2014) 285--313.

\bibitem{BrazG}{} J. Brazas, Generalized covering space theories, Theory Appl. Categ. 30 (2015) 1132--1162.

\bibitem{BrazFa}{} J. Brazas, P. Fabel, On fundamental group with the quotient topology, Homotopy Rel. Struc. 10 (2015) 71--91.

\bibitem{Baction}{} N. Brodskiy, J. Dydak, B. Labuz, A. Mitra, Group actions and covering maps in the uniform
category, Topol. Appl. 157 (2010) 2593--2603.

\bibitem{Bcovering}{} N. Brodskiy, J. Dydak, B. Labuz, A. Mitra, Covering maps for locally path connected spaces, Fund. Math. 218 (2012) 13--46.

\bibitem{BroU}{} N. Brodskiy, J. Dydak, B. Labuz, A. Mitra, Topological and uniform structures on universal covering spaces, arXiv:1206.0071.

\bibitem{Conner}
 G. Conner, M. Meilstrup, D. Repov\v{s,} A. Zastrow, M. \v{Z}eljko, On small homotopies of loops, Topol. Appl. 155 (2008) 1089--1097.

\bibitem{Fabel}
P. Fabel, The fundamental group of the harmonic archipelago, preprint, http://front.math.ucdavis.edu/math.AT/0501426.

\bibitem{Fischerexam}
H. Fischer, D. Repov\v{s}, \v{Z}. Virk, A. Zastrow, On semilocally simply connected spaces, Topol. Appl. 158 (2011) 397--408.

\bibitem{Zastrow}
H. Fischer, A. Zastrow, Generalized universal covering spaces and the shape group, Fund. Math. 197 (2007) 167--196.

\bibitem{FischerC}
H. Fischer, A. Zastrow, A core-free semicovering of the Hawaiian Earring, Topol. Appl. 160 (2013) 1957--1967.

\bibitem{Fox}
R.H. Fox, On shape, Fund. Math. 74 (1972) 47--71.

\bibitem{H}
P. Hilton, S. Wylie, Homology Theory: An introduction to algebraic topology, Cambridge University Press, 1960.

\bibitem{Lubkin}
S. Lubkin, Theory of covering spaces, Trans. Amer. Math. Soc. 104 (1962) 205--238.

\bibitem{Spanier}
E.H. Spanier, Algebraic Topology, McGraw-Hill, New York, 1966.

\bibitem{TorabiS}
H. Torabi, A. Pakdaman, B. Mashayekhy, On the Spanier groups and covering
and semicovering spaces, arXiv:1207.4394.

\bibitem{Torabi}
H. Torabi, A. Pakdaman,  B. Mashayekhy, Topological Fundamental Groups and Small Generated coverings,  Math. Slovaca 65 (2015) 1153--1164.

\bibitem{Virk}
\v{Z}. Virk, Small loop spaces, Topol. Appl. 157 (2010) 451--455.

\bibitem{ZVirk}
\v{Z}. Virk, Homotopical smallness and closeness, Topol. Appl. 158 (2011) 360-378.

\bibitem{VZexam}
\v{Z}. Virk, A. Zastrow, A homotopically Hausdorff space which does not admit a generalized universal covering space, Topol. Appl. 160 (2013) 656--666.

\bibitem{VZcom}
\v{Z}. Virk, A. Zastrow, The comparison of topologies related to various concepts of generalized covering spaces, Topol. Appl. 170 (2014) 52--62.




\end{thebibliography}



\end{document}